\documentclass[12pt,oneside,reqno]{amsart}
\usepackage[all]{xy}
\usepackage{amsfonts,amsmath,oldgerm,amssymb,amscd}
\UseComputerModernTips
\numberwithin{equation}{section}
\usepackage[breaklinks]{hyperref}

\usepackage{caption} 
\newtheorem{theorem}{Theorem}
\newtheorem{proposition}[theorem]{Proposition}

\newtheorem{lemma}[subsection]{{\bf Lemma}}

\newcommand{\al}{\alpha}
\newcommand{\be}{\beta}
\newcommand{\ga}{\gamma}

\newcommand{\Q}{\mbox{$\mathbb Q$}}

\begin{document}
\title[Diophantine equation of the form $U_{n_{1}} + U_{n_{2}} + \cdots + U_{n_{t}} = b_1 p_1^{z_1} + \cdots+ b_s p_s^{z_s}$]{Linear combinations of prime powers in sums of terms of binary recurrence sequences} 

\author[Meher]{N. K. Meher}
\address{Nabin Kumar Meher,  Harish-Chandra Research Institute (HBNI)\\ Chhatnag Road, Jhunsi\\ India, 211019}
\email{mehernabin@gmail.com}

\author[Rout]{S. S. Rout}
\address{Sudhansu Sekhar Rout, Departament of Mathematics\\ Harish-Chandra Research Institute (HBNI)\\ Chhatnag Road, Jhunsi\\ India, 211019}
\email{lbs.sudhansu@gmail.com}

\maketitle
\pagenumbering{arabic}
\pagestyle{headings}

\begin{abstract}
Let $\{ {U_{n}\}_{n \geq 0} }$ be a non-degenerate binary recurrence sequence with positive discriminant. Let $\{p_1,\ldots, p_s\}$ be fixed prime numbers and $\{b_1,\ldots ,b_s\}$ be fixed non-negative integers. In this paper, we obtain the finiteness result for the solution of the Diophantine equation  $U_{n_{1}} +  \cdots + U_{n_{t}} = b_1 p_1^{z_1} + \cdots+ b_s p_s^{z_s}
 $  under certain assumptions.  Moreover,  we explicitly solve the  equation $F_{n_1}+ F_{n_2}= 2^{z_1} +3^{z_2}$,  in  non-negative integers $n_1, n_2, z_1, z_2$  with $z_2\geq z_1$. The main tools used in this work are  the lower bound for linear forms in logarithms and the Baker-Davenport reduction method. This work generalizes the recent work of Mazumdar and Rout  \cite{Mazumdar} and Bert\'{o}k, Hajdu, Pink and R\'{a}bai \cite{Bertok2017}. 

\end{abstract}

\section{Introduction}

The problem of finding specific terms of a linear recurrence sequence of some particular form has a very rich history.  Peth\H{o} \cite{Petho1982} and Shorey-Stewart \cite{Shorey1983} independently studied the perfect powers in linear recurrence sequences. In particular, they consider the Diophantine equation 
\begin{equation}\label{eq0a}
U_{n}  = x^{z}
\end{equation}
in integers $n,x,z$ with $z\geq 2$, where $\{U_{n}\}_{n\geq 0}$ is a linear recurrence sequence and proved under certain natural assumptions that \eqref{eq0a} contain only finitely many perfect powers. For example, Fibonacci and Lucas numbers,   respectively,  of  the  form $x^{z}$, with $z > 1$ has been recently proved by Bugeaud, Mignotte and Siksek \cite{Bugeaud2006}. Similarly, finding all perfect powers in Pell sequence studied in \cite{Petho1991} (see also \cite{Cohn1996}). There are several results on Fibonacci numbers of particular form. Here we have cited few of them.  Literature for Fibonacci numbers of the forms $px^{2}+1, px^{3}+1, k^2 + k + 1, p^{a} \pm p^{b} + 1, p^{a} \pm p^b$ and $y^t \pm 1$ can be found in \cite{Bugeaud2008, Luca1999, Luca2007, Luca2009, Robbins1986}. Peth\H{o} and Tichy \cite{Petho1993} proved that there are only finitely  many  Fibonacci  numbers  of  the  form $p^{a} + p^b + p^c$, with $p$ prime. Marques and Togb\'{e} \cite{Marques2013} found all Fibonacci and Lucas numbers of the form $2^a + 3^b + 5 ^c$.

Recently,  Bert\'{o}k et al., \cite{Bertok2017} under some mild conditions gave finiteness result for the solutions of the Diophantine equation 
\begin{equation}\label{eq0b}
U_{n}  = b_1 p_1^{z_1} + \cdots+ b_s p_s^{z_s}
\end{equation}
in non-negative integers $z_1, \ldots, z_s$, where $\{U_{n}\}_{n\geq 0}$ is a binary non-degenerate recurrence sequence of positive discriminant, $p_{1}, \ldots, p_{s}$ are given primes and $b_1, \ldots, b_s$ are fixed integers. 

\smallskip

Several authors studied the problem to find $(n, m, z)$ such that
\begin{equation}\label{eq1}
U_{n} + U_{m} = 2^{z},
\end{equation}
where $\{U_{n}\}_{n\geq 0}$ is a fixed linear recurrence sequence. In the case of some special famous sequences, \eqref{eq1} have been studied \cite{BravoLuca2016, Bravo2014,  BL2015, Marques2014}. In \cite{Bravo2015}, Bravo et al., investigated the Diophantine equation 
\begin{equation}\label{eq1a}
F_{n} + F_{m} + F_{l} = 2^{z},
\end{equation}
in integers $(n, m, l, z)$, where $\{F_{n}\}_{n\geq 0}$ is a Fibonacci sequence. In \cite{Mazumdar}, Mazumdar and Rout extended equations \eqref{eq1} and \eqref{eq1a} and gave a finiteness result for the solutions of the Diophantine equation 
\begin{equation}\label{eq2}
U_{n_{1}} + U_{n_{2}} + \cdots + U_{n_{t}}  = p^{z}
\end{equation}
in non-negative integers $n_1, \ldots, n_t$ and $z$, where $\{U_{n}\}_{n\geq 0}$ is a binary non-degenerate recurrence sequence of positive discriminant and $p$ is a fixed prime.
   
\smallskip
   
In this article, we  generalize the results due to Bert\'{o}k et al., \cite{Bertok2017} and Mazumdar et al., \cite{Mazumdar} and consider the more general Diophantine equation  
\begin{equation}\label{eq3}
U_{n_{1}} + U_{n_{2}} + \cdots + U_{n_{t}} = b_1 p_1^{z_1} + \cdots+ b_s p_s^{z_s}
\end{equation}
in non-negative integers $n_1, \ldots, n_t, z_1,\ldots, z_s $,  where
$\{U_{n}\}_{n\geq 0}$  is a binary recurrence sequences of positive discriminant, $b_1, \ldots, b_s$ are fixed non-negative integers  and $p_1,\ldots, p_s$ are given primes. Further,  we determine all solutions of  the Diophantine equation \eqref{eq3} for $t = s = 2$  in non-negative integers $n_1, n_2, z_1, z_2$  with $z_2\geq z_1$ and the sequence $\{U_{n}\}_{n \geq 0}$ is the classical Fibonacci sequence.  To prove our main result,  we use  the lower bounds for linear form in logarithms  of algebraic numbers and a version of Baker - Davenport  reduction method.

\section{Notations and Main Results} 
  
The sequence $\{U_{n}\}_{n \geq 0} = \{U_{n}(P, Q, U_{0}, U_{1})\}$ is called a binary linear recurrence sequence if the relation
\begin{equation}\label{eq4}
 U_{n} = PU_{n-1} + QU_{n-2} \;\;(n\geq 2)
\end{equation}
holds, where $PQ\neq 0, U_{0}, U_{1}$  are fixed rational integers and $|U_{0}| + |U_{1}| > 0$. Then for $n\geq 0$
\begin{equation}\label{eq5}
U_{n} = \frac{a\alpha^{n}-b\beta^{n}}{\alpha-\beta}\quad (\alpha \neq \beta), 
\end{equation}
where $\alpha$ and $\beta$ are the roots of the polynomial $x^2-Px-Q$ and $a = U_{1} - U_{0}\beta,\, b = U_{1} - U_{0}\alpha$. The
sequence $\{U_{n}\}_{n\geq 0}$ is called non-degenerate, if $ab\al\be \neq 0$ and $\al/\be$ is not a root of unity. Taking $P = Q = 1, U_{0} = 0$ and $U_{1} = 1$, the sequence $\{U_{n}\}_{n \geq 0}$ becomes the classical Fibonacci sequence and usually denoted by $\{F_{n}\}_{n \geq 0}$ and the corresponding $\al = (1 + \sqrt{5})/2$ and $\be = (1 - \sqrt{5})/2$.

Throughout the paper, we assume that $U_{n}$ is non-degenerate, $\sqrt{\Delta} = (\alpha-\beta) >0$. The latter assumption implies that the sequence $\{U_{n}\}_{n \geq 0}$ have a dominant root and hence we can assume that $|\al| > |\be|$.

 In this paper, we prove the following. 
 
\begin{theorem}\label{th1}
Let $\{U_{n}\}_{n \geq 0}$ be a non-degenerate binary recurrence sequence with positive discriminant $\Delta > 0$.  Let $  p_1 \leqslant p_2 \leqslant \ldots \leqslant p_s   $ be given, not necessarily distinct prime numbers and  $b_1, \ldots, b_s$   be  fixed non-zero positive  integers. Put $K =  \max_{1\leq i \leq s}b_i $. Let $ 0 < \epsilon < 1$, and write $ T_{\epsilon}$ for the set of those solutions $ (n_1,\ldots, n_t, z_1, \ldots, z_s ) $ of \eqref{eq3}, for which $z_s = \max_{1 \leq i \leq s} z_i$  and further $z_i < (1 - \epsilon) z_s$ if $p_i = p_s$ for $1 \leq i \leq s-1$.  Then there exists an effectively computable constant $C$ depending only on $\epsilon, \{U_{n}\}_{n \geq 0}, K,$ $t, s, p_s$ such that  all solutions $(n_1,\ldots, n_t, z_1, \ldots, z_s)$ in $T_{\epsilon}$ satisfy
\[\max\{n_1,\ldots, n_t, z_1, \ldots, z_s\} < C.\]
\end{theorem}
Our next theorem illustrates Theorem \ref{th1},  in which we explicitly determine all the solutions of the Diophantine equation $F_{n_1}+ F_{n_2}= 2^{z_1} +3^{z_2}$. 
\begin{theorem} \label{th2}
The solutions of the Diophantine equation 
\begin{equation}\label{eq01}
F_{n_1}+ F_{n_2}= 2^{z_1} +3^{z_2}
\end{equation}
in  non-negative integers $  n_1, n_2, z_1, z_2$ with $ n_1 \geq n_2\geq 0 $  and $ z_2 \geq z_1$ are 
\begin{equation}\label{eq0c}
(n_1, n_2, z_1, z_2) \in \left\{\begin{array}{lr}
 (3, 0, 0, 0), (5, 0, 1, 1), (7, 0, 2, 2), (11, 0, 3, 4), (1, 1, 0, 0),\\
(2, 1, 0, 0), (2, 2, 0, 0), (3, 3, 0, 1), (4, 1, 0, 1),(4, 2, 0, 1),\\
 (4, 3, 1, 1), (5, 5, 0, 2), (6, 3, 0, 2), (6, 4, 1, 2), (6, 5, 2, 2),\\
 (8, 6, 1, 3), (9, 1, 3, 3), (9, 2, 3, 3), (10, 9, 3, 4), 
  (11, 6, 4, 4).
 \end{array}\right\}
 \end{equation}  
\end{theorem}

\section{Auxiliary results}
 
 Before proceeding further, we recall some facts and tools which we will use later.  We deduce the bound of $F_n$ from the well known Binet form for Fibonacci sequence as
 \begin{equation}\label{eq02}
 \alpha^{n-2} \leq F_n \leq \alpha^{n-1}.
 \end{equation}

\begin{lemma}\label{lem1}
There exist constants $c_0$ and $c_1$ such that following holds.
\begin{enumerate}
\item   For  $c_{0} = (|a| + |b|)/\sqrt{\Delta}$,
\begin{equation}\label{eq6}
|U_{n}|\leq c_{0}|\al|^{n}.
\end{equation}
\item  For $c_1 = \max\left\{\frac{2\log|\al|}{\log p_s}, \frac{\log|\al|(\log p_s + \log p_1)}{2\log p_1\log p_s}\right\}$,
\begin{equation}\label{eq7}
z_s \leq c_1 n_1.
\end{equation} 
\end{enumerate}
\end{lemma}
\begin{proof}
For proof of the inequality \eqref{eq6} one can refer \cite{Mazumdar}. Since $b_i$ are positive integers for $1\leq i \leq s$, we deduce from \eqref{eq3} that
\begin{align}\label{eq9}
\begin{split}
p_s^{z_s} \leq |b_1 p_1^{z_1} + \cdots + b_{s}p_{s}^{z_{s}}| &= |U_{n_1} + \cdots + U_{n_t}| \leq  |U_{n_1}|+ \cdots + |U_{n_t}|\\
&\leq tc_0|\al|^{n_1} 
 = p_s^{\frac{ \log (t c_0) + n_1 \log |\alpha|}{\log p_s}}. 
 \end{split}
\end{align}
If 
\[\log (t c_0) >  n_1 \log |\alpha|,\]
then from \eqref{eq9}, $z_s$ is bounded  and hence Theorem \ref{th1} follows trivially. Thus we have,
\[ z_s \leq \frac{2n_1 \log |\alpha|}{\log p_s}.\]
Further, for better lower bound, we may take 
\[\log (t c_0) >  \frac{n_1 \log |\alpha|(\log p_s - \log p_1)}{2\log p_1},\]
and in this case, Theorem \ref{th1} also easily follows.
Thus, assuming on contrary and using \eqref{eq9}, we get
\begin{equation}\label{eq10}
z_s \leq \frac{n_1 \log |\alpha|(\log p_s + \log p_1)}{2\log p_1 \log p_s}.
\end{equation}
This proves the inequality in \eqref{eq7}.
\end{proof}

Let $ \alpha $ be an algebraic number of degree $d$. Then the \textit{ logarithmic height}  of the algebraic number $\alpha$ is given  by
$$ 
h(\alpha) = \frac{1}{d} \left( \log |a| + \sum_{i=1}^d \log \max ( 1, |\alpha^{(i)}| ) \right),
$$
where $a$ is the leading coefficient of the minimal polynomial of $\alpha$ and the $\alpha^{(i)}$'s are the conjugates of $\alpha$ in $\mathbb{C}$.

To prove Theorem \ref{th1}, we use lower bounds for linear forms in logarithms to bound the index $n_1$ appearing in \eqref{eq3}. 
We need the following general lower bound for linear forms in logarithms due to Matveev \cite{Matveev2000}.
\begin{lemma}[Matveev \cite{Matveev2000}]\label{lem12}
Let $\ga_1,\ldots,\ga_t$ be real algebraic numbers and let $b_{1},\ldots, b_{t}$ be rational integers. Let $D$ be the degree of the number field $\mathbb{Q}(\ga_1,\ldots,\ga_t)$ over $\mathbb{Q}$ and let $A_{j}$ be real numbers satisfying 
\[ A_j \geq \max \left\{ Dh(\ga_i) , |\log \ga_i|, 0.16  \right\}, \quad j= 1, \ldots,t.\]
Assume that $B\geq \max\{|b_1|, \ldots, |b_{t}|\}$ and $\Lambda:=\ga_{1}^{b_1}\cdots\ga_{t}^{b_t} - 1$. If $\Lambda \neq 0$, then
\[|\Lambda| \geq \exp \left( -1.4\times 30^{t+3}\times t^{4.5}\times D^{2}(1 + \log D)(1 + \log B)A_{1}\cdots A_{t}\right).\]

	\end{lemma}

After getting the upper bound of $n_1$ which is generally too large, the next step is to reduce it. For this reduction purpose, we use a variant of the Baker-Davenport result \cite{Baker1969}. Here, for a real number $x$, let $||x|| := \min \{|x - n| : n \in \mathbb{Z}\}$ denote the distance from $x$ to the nearest integer.

\begin{lemma}[\cite{Dujella1998}]\label{lem13}
Suppose that $M$ is a positive integer, and $A, B$ are positive reals with $B > 1$. Let $p/q$ be the convergent of the continued fraction expansion of the irrational number $\ga$ such that $q > 6M$, and let  $\epsilon := ||\mu q|| - M||\ga q||$, where $\mu$ is a real number. If  $\epsilon > 0$, then there is no solution of the inequality
\[0 < u \ga - n + \mu < AB^{-m}\]
in positive integers $u, m$ and $n$ with
\[u \leq M\quad \mbox{and}\quad m\geq \frac{\log (Aq/\epsilon)}{\log B}.\]
\end{lemma}

Estimating the lower bounds for the height of elements in a number field of given degree is tough. However, for the quadratic fields we have the following lemma due to Pink and Ziegler \cite{Pink2016}.
\begin{lemma}[\cite{Pink2016}]\label{lem10} 
Let $\al$ is an algebraic number of degree two. Then $h(\al) \geq 0.24$ or is a root of unity.
\end{lemma}

In order to apply Matveev's Theorem,  we must ensure that $\Lambda$ does not vanish. In this regard, we have the following lemma.
\begin{lemma}\label{lem13a}
Suppose $\Lambda_i := p_s^{z_s}\alpha^{-n_1} b_s a^{-1}\sqrt{\Delta}(1 + \al^{n_2 - n_1} + \cdots + \al^{n_i - n_1})^{-1} - 1$ for all $1\leq i \leq t$. Let $r = \sqrt{\Delta}$ and $c_1$ be in Lemma \ref{lem1}(2). If  $\Lambda_i = 0$, then 
\[n_1 \leq \max\left\{\frac{\log (i|b/a|)}{\log (|\al/\be|)}, \frac{\log (i|b/a|)}{\log |\al|}, \frac{\log (b_s|r|/|a|)}{\log |\al|/p_s^{c_1}}\right\}.\]
 \end{lemma}
\begin{proof}
Proof follows from \cite[Lemma 3.5]{Mazumdar}.
\end{proof}
The following lemma gives a relation between height of an algebraic number and its logarithm. It is useful when we apply  Lemma \ref{lem12} in the proof of Theorem \ref{th1}.

\begin{lemma}[\cite{Mazumdar}]\label{lem13b}
Let $\gamma_i := a^{-1}\sqrt{\Delta}(1 + \al^{n_2 - n_1} + \cdots + \al^{n_i - n_1})^{-1}$ be the algebraic numbers in the number field $\Q(\sqrt{\Delta})$ for $1 \leq i \leq t$. Then 
\begin{align}\label{eq11}
\begin{split}
A(i): &= 2 \log a + \log \Delta + 2\left(|n_2 - n_1| + \cdots + |n_i - n_1|\right) \log |\al| +  i\log4
\\ & >\max\{2h(\ga_i),  |\log \ga_i|\}.
\end{split}
\end{align}
\end{lemma}

Suppose $\ell_i:= \max\left\{\frac{\log (i|b/a|)}{\log (|\al/\be|)}, \frac{\log (i|b/a|)}{\log |\al|}, \frac{\log (b_s|r|/|a|)}{\log(|\al|/p_s^{c_1})}\right\}$ (see Lemma \ref{lem13a}). We denote $\ell:= \max\{\ell_1, \cdots, \ell_{t}\}$.
If $n_1 \leq \ell$, then the conclusion of Theorem \ref{th1} follows trivially. Thus we  may assume that $n_1 > \ell$.

Now we are ready to prove Theorem \ref{th1}. The proof is  some how motivated by \cite{BL2015, Mazumdar}.

\section{Proof of Theorem \ref{th1}}
Notice that if $n_t = 0$ and $(t-1) \geq 2$, then it is equivalent to consider \eqref{eq3} again. If $n_t = 0$ and $t = 2$ then \eqref{eq3} reduces to $U_{n_{1}} = b_1 p^{z_1} + \cdots+ b_{s} p^{z_{s}}$ and hence theorem follows from \cite{Bertok2017}. Suppose $n_1 = \cdots = n_t$, then \eqref{eq3} becomes $tU_{n_{1}} = b_1 p^{z_1} + \cdots+ b_{s} p^{z_{s}}$. In this case, theorem also follows from \cite{Bertok2017} by replacing $b_i$ with $b_i/t$ for $1\leq i \leq s$. From now on, we assume $n_1 > n_2 > \cdots > n_{t}$. Indeed, if some of the $n_i$'s are equal, we can group them together obtaining a representation of $U_{n}$ of the from
\[a_1 U_{n_1} + a_2 U_{a_{1} + 1} + \cdots + a_r U_{a_{r - 1} + 1 } = b_1 p^{z_1} + \cdots+ b_{s} p^{z_{s}}\]
 with $a_1 + \cdots + a_r = t$ and this equation can be handled similar to that of \eqref{eq3} with some changes in constants. 

\subsection{Bounding $(n_1 - n_i)$ in terms of $n_1$ for $2\leq i \leq t$} \label{sec4.1}

First we estimate an upper bound for $b_1 p^{z_1} + \cdots+ b_{s - 1} p^{z_{s - 1}}$. In fact, this bound was obtained in \cite{Bertok2017}. For the sake of completeness, we write it again here. Put $\delta_1 = \epsilon$, if $p_1 = p_2 = \cdots = p_s$. Since $p_1 \leqslant \cdots \leqslant p_s$ and $z_s = \max_{1\leq i\leq s}z_i$, we get 
\begin{align}\label{eq12}
\begin{split}
|b_1 p_1^{z_1} + \cdots+ b_{s - 1} p_{s - 1}^{z_{s - 1}}| & = p_s^{z_s} \left|b_1\frac{p_1^{z_1}}{p_s^{z_s}} + \cdots + b_{s - 1}\frac{p_{s - 1}^{z_{s - 1}}}{p_s^{z_s}}\right| \\
& \leq (s - 1) K p_s^{(1 - \delta_1)z_s},
\end{split}
\end{align}
where $K = \max_{1\leq i\leq s} b_i$ and \[\delta_1 = \min\left(\epsilon, 1- \max_{p_i <p_s}(\log p_i/\log p_s)\right).\]
Here, we claim that for $2\leq i \leq t$, we have
\begin{equation}\label{eq12a}
n_1 - n_i \leq C_i (\log n_1)^{i - 1},
\end{equation}
where $C_i$'s are effectively computable constants depending on  $\epsilon, \{U_{n}\}_{n \geq 0}, K,$ $t, s, p_s$. 
We use induction on $i$ to find an upper bound of $n_1 - n_i$ for $2\leq i \leq t$. First, we calculate the upper bound of $n_1 - n_2$.
Rewrite \eqref{eq3} as
\begin{equation}\label{eq13}
\frac{a\al^{n_1}}{\sqrt{\Delta}} - b_s p_s^{z_s} =  \frac{b\be^{n_1}}{\sqrt{\Delta}} + b_1 p_1^{z_1} + \cdots+ b_{s - 1} p_{s - 1}^{z_{s - 1}} - (U_{n_2} + \cdots + U_{n_t}).
\end{equation}
Taking absolute value on both sides of \eqref{eq13} and using the inequality \eqref{eq12}, we obtain  
\begin{align}\label{eq14}
\begin{split}
\left|\frac{a\al^{n_1}}{\sqrt{\Delta}} - b_s p_s^{z_s}\right| 
& \leq \left|\frac{b\be^{n_1}}{\sqrt{\Delta}}\right| + |b_1 p_1^{z_1} + \cdots+ b_{s - 1} p_{s - 1}^{z_{s - 1}}| + |U_{n_2} + \cdots + U_{n_t}|\\
& \leq \left|\frac{b\be^{n_1}}{\sqrt{\Delta}}\right| + (s - 1) K p_s^{(1 - \delta_1)z_s} + (t-1)c_{0}|\al|^{n_{2}}.
\end{split}
\end{align}
Now dividing both sides of the inequality  \eqref{eq14} by $a\al^{n_1}/\sqrt{\Delta}$, we get
\begin{equation}\label{eq15}
\left|1 - b_s p_s^{z_s}\alpha^{-n_1} a^{-1}\sqrt{\Delta} \right|  \leq \left|\frac{b\be^{n_1}}{a \al^{n_1}}\right| + \frac{(s - 1)\sqrt{\Delta} K p_s^{(1 - \delta_1)z_s}}{|a||\alpha|^{n_1}} + \frac{(t-1)c_{0}\sqrt{\Delta}|\al|^{n_{2}}}{|a||\al|^{n_{1}}}.
\end{equation}
First we estimate the middle term in right hand side of \eqref{eq15}. From Lemma \ref{lem1}(2), 
\begin{align}\label{eq16}
\begin{split}
\frac{p_s^{(1 - \delta_1)z_s} }{|\alpha|^{n_1}} & = |\alpha|^{\frac{(1 - \delta_1)z_s\log p_s}{\log |\alpha|} - n_1}\\
& \leq |\alpha|^{\frac{(1 - \delta_1)\log p_s}{\log |\alpha|} \frac{n_1 \log |\alpha|(\log p_s + \log p_1)}{2\log p_1 \log p_s} - n_1}\\
& = |\alpha|^{n_1\left(\frac{(1 - \delta_1)(\log p_s + \log p_1)}{2\log p_1} - 1\right)}.
\end{split}
\end{align}
If $\delta_1 = 1- \max_{p_i <p_s}(\log p_i/\log p_s)$ then 
\begin{equation}\label{eq17}
1 -  \delta_1 = \min_{p_i <p_s}(\log p_i/\log p_s) = \log p_1 /\log p_s.
\end{equation}
Putting the value of $1 -  \delta_1$ in the inequality \eqref{eq16}, we have
\begin{equation}\label{eq18}
\frac{p_s^{(1 - \delta_1)z_s} }{|\alpha|^{n_1}} \leq |\alpha|^{n_1\left(\frac{\log p_s + \log p_1}{2\log p_s} - 1\right)} = |\alpha|^{- n_1 \delta_2},
\end{equation}
where $\delta_2 = (1 - \frac{\log p_s + \log p_1}{2\log p_s}) > 0$ .

Now, suppose that $\delta_1 = \epsilon$. In this case, $p_1 = p_2 = \cdots = p_s$ and hence the inequality \eqref{eq16} becomes
\begin{equation}\label{eq19}
\frac{p_s^{(1 - \delta_1)z_s} }{|\alpha|^{n_1}} =\frac{p_s^{(1 - \epsilon)z_s} }{|\alpha|^{n_1}} = \frac{|\alpha|^{ (1 - \epsilon)z_s \frac{\log p_s }{\log |\alpha|} }}{|\alpha|^{n_1}} \leq |\alpha|^{n_1( 1 - \epsilon - 1)}  \leq |\alpha|^{- \epsilon n_1 }.
\end{equation}
Altogether, from inequalities \eqref{eq18} and \eqref{eq19}, we get
\begin{equation}\label{eq20}
\frac{p_s^{(1 - \delta_1)z_s} }{\alpha^{n_1}} \leq \alpha^{- \delta_3 n_1}, 
\end{equation}
where $\delta_3 = \min (\epsilon, \delta_2)$.
From inequalities \eqref{eq15} and \eqref{eq20} with $|\beta| > 1$, we have
\begin{align}\label{eq21}
\begin{split}
&\left|1 - b_s p_s^{z_s}\alpha^{-n_1} a^{-1}\sqrt{\Delta} \right| \\ 
& \leq \frac{|b|}{|a|}\left(\frac{|\al|}{|\be|}\right)^{n_2 - n_1} + \frac{(s - 1)\sqrt{\Delta} K}{|a|} |\al|^{\delta_3(n_{2} -n_1)}+ \frac{(t-1)c_{0}\sqrt{\Delta}}{|a|}|\al|^{n_{2} -n_1}\\
&\leq \frac{|b|}{|a|}\left(\frac{|\be|}{|\al|}\right)^{n_1 - n_2} + \frac{(s - 1)\sqrt{\Delta} K}{|a|} \left(\frac{|\be|}{|\al|}\right)^{\delta_3(n_{1} - n_2)} +\frac{(t-1)c_{0}\sqrt{\Delta}}{|a|}\left(\frac{|\be|}{|\al|}\right)^{n_{1} -n_2}.
\end{split}
\end{align}
As $0 < \delta_3 < 1$, 
\begin{equation}\label{eq22}
\left|1 - p_s^{z_s}\alpha^{-n_1} b_s a^{-1}\sqrt{\Delta} \right| \leq \left(\frac{|b|}{|a|} + \frac{(s - 1)\sqrt{\Delta} K}{|a|} + \frac{(t-1)c_{0}\sqrt{\Delta}}{|a|}\right)\left(\frac{|\be|}{|\al|}\right)^{\delta_3(n_{1} - n_2)}.
\end{equation}
By the same line of argument, when $|\beta| \leq 1$, we have
\begin{equation}\label{eq22a}
\left|1 - p_s^{z_s}\alpha^{-n_1} b_s a^{-1}\sqrt{\Delta} \right| \leq \left(\frac{|b|}{|a|} + \frac{(s - 1)\sqrt{\Delta} K}{|a|} + \frac{(t-1)c_{0}\sqrt{\Delta}}{|a|}\right)|\al|^{-\delta_3(n_{1} - n_2)}.
\end{equation}

Thus, for any $\beta$ it follows that
\begin{equation}\label{eq23}
\left|1 - p_s^{z_s}\alpha^{-n_1} b_s  a^{-1}\sqrt{\Delta} \right| \leq \frac{|c_3|}{\min \left(\frac{|\al|}{|\be|}, |\al|\right)^{\delta_3(n_{1} - n_2)}}, 
\end{equation}
where $|c_3| = \left(\frac{|b|}{|a|} + \frac{(s - 1)\sqrt{\Delta} K}{|a|} + \frac{(t-1)c_{0}\sqrt{\Delta}}{|a|}\right)$. 
In order to apply Lemma \ref{lem12}, we take 
\[\gamma_1 := p_s,\; \gamma_2 := \alpha,\; \gamma_3 := b_s a^{-1}\sqrt{\Delta},\; b_1 := z_s,\; b_2 := -n_1,\; b_3= 1.\] 
Thus, our first linear form is $\Lambda_1 := \ga_{1}^{b_{1}}\ga_{2}^{b_{2}}\ga_{3}^{b_{3}} - 1$ and  $\Lambda_1 \neq 0$ from Lemma \ref{lem13a}.  Here we are taking the field $\Q(\sqrt{\Delta})$ over $\Q$ and $t=3$. Finally, we recall from Lemma \ref{lem1}(2),  $z_s\leq c_1n_1$ and deduce that 
\[\max\{|b_1|, |b_2|, |b_3|\} = \max\{z_s, n_1, 1\} \leq c_1'n_1,\]
where $c_1' = \max\{c_1, 1\}$. Hence we can take $B := c_1'n_1$. Also $D= 2, h(\gamma_1)= \log p_s, h(\gamma_2)\leq (\log \alpha)/2, h(\gamma_3)\leq \log (a b_s) + 2\log \Delta$. Thus, we can take $A_1 := 2\log p_s, A_2 = \log \alpha, A_3 = \log b_s + \log a + 2\log \Delta$.  Employing Lemma \ref{lem12}, we have 
\begin{align*}
|\Lambda_1| & > \exp \left(-C_0(1+ \log c_1'n_1)\right),
\end{align*}
where $C_0:= 1.4\times 30^{6}\times 3^{4.5}\times 4\times (1+\log 2)(2\log p_s) (\log \alpha) (\log a +\log b_s + 2 \log \Delta)$. So the above inequality can be rewritten as, 
\begin{equation}\label{eq24}
\log |\Lambda_1| > - C_{1} \log n_1. 
\end{equation} 
Taking logarithms in  the inequality \eqref{eq23} and comparing the resulting inequality with \eqref{eq24}, we get 
\begin{equation}\label{eq25}
(n_1 - n_2) < C_2 \log n_1.
\end{equation}
Therefore, for $i = 2$, the statement is true. By induction hypothesis, we  may assume that \eqref{eq12a} is true for $i =2, \ldots, t-1$.  We want to  bound $n_1 - n_t$. So, we formulate the required linear form as follows. Again, rewrite \eqref{eq3} as
\begin{align*}
\left|\frac{a\al^{n_1}}{\sqrt{\Delta}} + \cdots + \frac{a\al^{n_{t - 1}}}{\sqrt{\Delta}} - b_sp_s^{z_s}\right|  & = \left|\frac{b\be^{n_1}}{\sqrt{\Delta}} + \cdots + \frac{b\be^{n_{t - 1}}}{\sqrt{\Delta}} + b_1p_1^{z_1} + \cdots b_{s - 1}p_{s - 1}^{z_{s - 1}} + U_{n_{t}} \right |.
\end{align*}
Using  the triangle inequality and the inequality \eqref{eq12}, we have
\begin{align}\label{eq26}
\left|\frac{a\al^{n_1}}{\sqrt{\Delta}}(1 + \al^{n_2 - n_1} + \cdots + \al^{n_{t - 1} - n_1}) - b_sp_s^{z_s}\right|& \leq  \frac{(t - 1)b|\be|^{n_1}}{\sqrt{\Delta}} +  (s - 1) K p_s^{(1 - \delta_1)z_s} + c_0 |\al|^{n_{t}}.
\end{align}
Dividing both sides of the inequality \eqref{eq26} by $a\al^{n_1}(1 + \al^{n_2 - n_1} + \cdots + \al^{n_{t - 1}- n_1})/\sqrt{\Delta}$ and using the inequality \eqref{eq20}, we get the linear form as
\begin{equation}\label{eq27}
\left|1 - b_sp_s^{z_s}\alpha^{-n_1} a^{-1}\sqrt{\Delta}(1 + \al^{n_2 - n_1} + \cdots + \al^{n_{t - 1} - n_1})^{-1} \right|  \leq \frac{|c_{5}|}{\min \left(\frac{|\al|}{|\be|}, |\al|\right)^{ \delta_3 (n_{1} -n_t)}}.
\end{equation}
In order to apply Lemma \ref{lem12},  we take the following parameters
\begin{align*}
\gamma_1 := p_s,\; &\gamma_2 := \alpha,\; \gamma_3 := b_sa^{-1}\sqrt{\Delta}(1 + \al^{n_2 - n_1} + \cdots + \al^{n_{t - 1} - n_1})^{-1},\\
& b_1 := z_s,\;\; b_2 := -n_1,\;\; b_3= 1.
\end{align*}
Therefore,  $\Lambda_2 := \ga_{1}^{b_{1}}\ga_{2}^{b_{2}}\ga_{3}^{b_{3}} - 1$. 
Also, $\Lambda_2 \neq 0$ from Lemma \ref{lem13a}. As $A_1, A_2$ are already estimated in previous case,  we  have to estimate only  $A_{3}$. Using Lemmas \ref{lem10} and \ref{lem13b} we can take \[A_3: = c_{6} + [(n_1 - n_2) + \cdots + (n_1 - n_{t-1})]c_7 > \max \{2h(\ga_{3}), |\log \ga_{3}|, 0.16\}.\]  Again, from the Lemma \ref{lem12} and the  inequality \eqref{eq27}, we have
\begin{equation}\label{eq28}
\exp \left(-c_{8}\log n_{1}(c_{6} + [(n_1 - n_2) + \cdots + (n_1 - n_{t-1})]c_7) \right) \leq \frac{|c_{5}|}{\min \left(\frac{|\al|}{|\be|}, |\al|\right)^{\delta_3(n_{1}-n_{t})}}.  
\end{equation}
Using the induction hypothesis, we obtain 
\begin{equation}\label{eq29}
(n_1 - n_t) < C_t (\log n_1)^{t - 1}.
\end{equation}
\subsection{Bounding  $n_1$} \label{sec4.2}
To bound $n_1$, we rewrite the equation \eqref{eq3} as  
\begin{equation*}
\frac{a\al^{n_1}}{\sqrt{\Delta}} + \cdots + \frac{a\al^{n_t}}{\sqrt{\Delta}} - b_sp_s^{z_s}  = \frac{b\be^{n_1}}{\sqrt{\Delta}} + \cdots + \frac{b\be^{n_t}}{\sqrt{\Delta}} + b_1p_1^{z_1} + \cdots + b_{s - 1}p_{s - 1}^{z_{s - 1}}, 
\end{equation*}
and then taking absolute value on both sides, we get
\begin{align*}
\left|\frac{a\al^{n_1}}{\sqrt{\Delta}}(1 + \al^{n_2 - n_1} + \cdots + \al^{n_t - n_1}) - b_sp_s^{z_s}\right|& \leq \left|\frac{b(\be^{n_1} + \cdots +\be^{n_t})}{\sqrt{\Delta}} \right| + (s - 1) K p_s^{(1 - \delta_1)z_s}.
\end{align*}
Now dividing through out by $\frac{a\alpha^{n_1}}{\sqrt{\Delta}}(1+\cdots+ \alpha^{n_t-n_1})$, we get
\begin{equation}\label{eq30}
\left|1 - b_sp_s^{z_s}\alpha^{-n_1} a^{-1}\sqrt{\Delta}(1 + \al^{n_2 - n_1} + \cdots + \al^{n_t - n_1})^{-1} \right|  \leq \frac{|c_9|}{\min \left(\frac{|\al|}{|\be|}, |\al|\right)^{\delta_3 n_{1}}}. 
\end{equation}
To apply  Lemma  \ref{lem12}, we take the following parameters.
\begin{align*}
&\gamma_1 := p_s, \gamma_2 := \alpha, \gamma_3 := b_s a^{-1}\sqrt{\Delta}(1 + \al^{n_2 - n_1} + \cdots + \al^{n_t - n_1})^{-1}\\
& b_1 := z_s,\;\; b_2 := -n_1,\;\; b_3= 1.
\end{align*}
Thus, the final linear form is $\Lambda_t := \ga_{1}^{b_{1}}\ga_{2}^{b_{2}}\ga_{3}^{b_{3}} - 1$ and is non-zero by Lemma \ref{lem13a}. From the conclusions of Lemma \ref{lem10} and \ref{lem13b}, we can take 
\[A_3: = c_{10} + c_{11}[(n_1 - n_2) + \cdots + (n_1 - n_t)] > \max \{2h(\ga_{3}), |\log \ga_{3}|, 0.16\}.\]
Using Lemma \ref{lem12} and the inequality \eqref{eq30}, we have
\begin{equation}\label{eq31}
\exp \left( - c_{12}\log n_{1}(c_{10} + c_{11}[(n_1 - n_2) + \cdots + (n_1 - n_t)]) \right) < \frac{|c_{9}|}{\min \left(\frac{|\al|}{|\be|}, |\al|\right)^{\delta_3n_{1}}}.  
\end{equation}
Putting the upper bounds of $n_1 - n_i$ for $2\leq i \leq t$ obtained  in Section \ref{sec4.1}, we get 
\begin{equation}\label{eq32}
n_1  < C_t (\log n_1)^{t},
\end{equation}
 and this completes the proof of Theorem \ref{th1}.

\smallskip

\section{Proof of Theorem \ref{th2}}
For the reason of symmetry in \eqref{eq01}, we assume that $n_1 \geq n_2$. Firstly, suppose that $n _1 = n_2$. In this case, if both $z_1 \neq 0$ and $z_2 \neq 0$, then \eqref{eq01} becomes 
\[2F_{n_1} = 2^{z_1} + 3^{z_2},\]
and this equation has no solution as left hand side is even and right hand side is odd. Next consider $z_1 = z_2 = 0$. In this case, \eqref{eq01} is $F_{n_1} = 1$ and this is possible only when $n_1 = 1, 2$. Hence, the solution $(n_1, n_2, z_1, z_2)$ of \eqref{eq01} is $(1, 1, 0, 0), (2, 1, 0, 0), (2, 2, 0, 0)$. Now assume that  $n _1 > n_2$. Further, if $n_2 = 0$, then  \eqref{eq01} becomes 
\[F_{n_1} = 2^{z_1} + 3^{z_2},\]
and the list of solutions is given in \cite[Table 1]{Bertok2017}. From now on, assume that $n _1 > n_2 > 0$. First we list all solutions of \eqref{eq01} with $n_1 \leq 100$. Here one can notice that $z_{2}$ is also bounded as $z_1 \leq z_2$ and
\[3^{z_2} \leq  2^{z_1} + 3^{z_2} = F_{n_1} + F_{n_2} \leq 2F_{100}.\]
To find all solutions \eqref{eq01}, we use a program written in {\sf Mathematica} in the range $0 < n_2 < n_1 < 100$ and $z_1 \leq z_2 \leq \log(2F_{100})/ \log 3$ and and the program returns the set of solutions, which are explicitly given in \eqref{eq0c}. So, from here onward, we work on the assumption that $n_1 > n_2$ and $ n_1 > 100$.  
 
\subsection{Bounding $(n_1 - n_2)$ in terms of $n_1$} 
From \eqref{eq01} and \eqref{eq02}, we get
\begin{equation}\label{eq33}
3^{z_2} \leq |2^{z_1} + 3^{z_2} | = |F_{n_1} + F_{n_2}|\leq 2|\al|^{n_1 - 1}.
\end{equation}
Taking logarithms on both sides of the inequality \eqref{eq33}, we obtain
\begin{equation}\label{eq34}
z_2 \leq \frac{\log 2}{\log 3} + (n_1 - 1)\frac{\log |\al|}{\log 3} \leq 0.45 n_1.
\end{equation}
We obtain the first linear form using $ |\beta| < 1 $ and $ z_1 \leq z_2 \leq 0.45  n_1$ similar to the inequality \eqref{eq23},
\begin{align}\label{eq35}
\begin{split}
&\left|1 -  3^{z_2}\alpha^{-n_1} \sqrt{5} \right|  
 \leq \left(\frac{1 }{|\alpha|}\right)^{n_1 } + \left(\frac{ 2^{z_1} \sqrt{5}}{|\alpha|^{n_1 }}\right) + \frac{\sqrt{5}}{|\alpha|}|\al|^{n_{2} -n_1}\\
&\leq \frac{3\sqrt{5}}{\left(\alpha/2^{0.45}\right)^{n_1-n_2}}.   
\end{split}
\end{align}
In order to apply Lemma \ref{lem12}, we take 
$B = n_1, h(\gamma_1) = \log 3 = 1.0986 < 1.1 , h(\gamma_2)=(\log \al)/2 = 0.2406 < 0.25, h(\gamma_3) = \log(\sqrt{5}) < 0.81$. We can choose $A_1 = 2.2, A_2 =  0.5, A_3 =  1.7$. Using these parameters, we obtain 
\[\exp( - 1.8 \times 10^{12}\times (1 + \log n_1)) < \frac{3\sqrt{5}}{\left(\alpha/2^{0.45}\right)^{n_1-n_2}}.\] 
Further, since $(1+\log n_1) < 2\log n_1$ for $n_1 > 100$, we get 
\begin{equation}\label{eq36}
(n_1-n_2)\log \left(\alpha/2^{0.45}\right) < 3.7\times 10^{12} \log n_1,
\end{equation}
which leads to
\begin{equation}\label{eq37}
(n_1-n_2) < \frac{3.7\times 10^{12}}{\log \left(\alpha/2^{0.45}\right)} \log n_1 < 2.18 \times 10^{13}\log n_1.
\end{equation}
Similarly, the inequality corresponding to  second linear form is 
\begin{equation}\label{eq38}
| 1- 3^{z_2} \sqrt{5} \alpha^{-n_1}(1 + \alpha^{n_2-n_1})^{-1}|\leq \frac{3\sqrt{5}}{\left(\alpha/2^{0.45}\right)^{n_1}}. 
\end{equation}
In this case, $A_1$ and $A_2$ are same as in the previous case, and since
\[2(\log \sqrt{5} + (n_1 - n_2)\frac{\log \al}{2} + \log 2) < 1.16 \times 10^{14} \log n_1, \]
we take $A_3 := 1.16 \times 10^{14} \log n_1$. Again employing Lemma \ref{lem12}, we obtain
\begin{equation}\label{eq39}
\exp(- 1.06 \times 10^{12}\times (1 + \log n_1)(1.16 \times 10^{14} \log n_1)) \leq  \frac{3\sqrt{5}}{\left(\alpha/2^{0.45}\right)^{n_1}}
\end{equation}
and this implies
\begin{equation}\label{eq40}
n_1  < \frac{2.47\times 10^{26}}{\log \left(\alpha/2^{0.45}\right)} (\log n_1)^2 < 1.45 \times 10^{27}(\log n_1)^2.
\end{equation}
Now, we summarize the above discussion in the following proposition. 
\begin{proposition}\label{prop1}
Let us assume that $n_1 > n_2$ and $n_1 > 100$. If
$(n_1, n_2, z_1, z_2)$ 
is  a positive integral solution of  equation \eqref{eq01},   then
\begin{equation}\label{eq41}
n_1 < 9\times 10^{30}. 
\end{equation}
\end{proposition}
\subsection{Reducing the size of $n_1$}
From Proposition \ref{prop1}, we can see that the bound  obtained for $n_{1}$ is very large. Now our job is to reduce this upper bound to a certain minimal range. From the inequality \eqref{eq35}, put 
\begin{equation}\label{eq42}
\Lambda _1 := z_2\log3 - n_1\log\alpha + \log \sqrt{5}.
\end{equation}
Then, 
\[|1- e^{\Lambda_1}| \leq \frac{3\sqrt{5}}{\left(\alpha/2^{0.45}\right)^{n_1-n_2}}.\] 
Here, we first claim that $\Lambda_1 \neq 0$. Suppose $\Lambda_1 = 0$, then
\begin{equation}\label{eq42a}
3^{z_2} \sqrt{5} = \al^{n_1}.
\end{equation} 
Taking squares on both sides of \eqref{eq42a}, we reach at a contradiction as  the left hand side of resulting equation is rational whereas the right hand side is irrational. Thus, we have  $\Lambda_1 \neq 0$.  Now, we consider  the cases $\Lambda_1 > 0$  and $\Lambda_1< 0$ separately. 

Let us suppose that $\Lambda_1 > 0$. Since $ x < e^x -1 $ holds for all positive real numbers $x$, we deduce 
 \begin{equation}\label{eq43}
 0 < \Lambda_1 < \frac{3\sqrt{5}}{(\al/2^{0.45})^{n_1 - n_2}}.
\end{equation}  
Dividing the inequality \eqref{eq43} by $\log \alpha$, we have
\begin{equation}\label{eq44}
0< z_2\left(\frac{\log3}{\log\alpha}\right)-n_1 +\left(\frac{\log\sqrt{5}}{\log \alpha}\right)< \frac{14}{(\al/2^{0.45})^{n_1 - n_2}}.
\end{equation}
We are now ready to use Lemma \ref{lem13} with the obvious parameters
\[ \gamma:= \frac{\log3}{\log\alpha},\;\; \mu:= \left(\frac{\log \sqrt{5}}{\log \alpha}\right),\;\;
 A:= 14,\;\; B:= (\al/2^{0.45}).\]
We claim that $\ga $ is irrational. In fact, if $ \ga = \frac{p}{q} $, then $ \alpha^{p} =3^q \in Q $ which is absurd.
   
Let $[a_0, a_1, a_2,a_3,\ldots] = [2, 7/3, 9/4, 16/7, \ldots]$ be a continued fraction expansion of $\ga$ and let $p_k/q_k$ be its $k$-th convergent. Take $M:= 9\times 10^{30}$, then using {\sf Mathematica}, one can see that 
\[ 6M < q_{64},\]
and hence we consider $\epsilon:= \|\mu q_{64}\| - M\|\ga q_{64}\|$, which is positive.  Then by Lemma \ref{lem13},  if \eqref{eq01} has a solution $(n_1, n_2, z_1, z_2)$, then $(n_1 - n_2) \in [0,485]$. 
Now, we  assume the other case, i.e., $\Lambda_1 < 0$. Here we consider two sub cases according as $(n_1 - n_2) > 15$ and $(n_1 - n_2) \leq 15$.  

\smallskip
{\bf Sub-case I.}( $(n_1 - n_2) \leq 15$).

Recall that, we have $n_1 > 100$ and considering this case, we get $n_2 \geq 85$. Since $\Lambda_1 < 0$, we get 
\begin{equation}\label{eq45a}
3^{z_2} < \frac{\alpha^{n_1}}{\sqrt{5}}.
\end{equation}
Now adding $ 2^{z_1} $ in both sides of \eqref{eq45a}, we have
\begin{align}\label{eq45b}
F_{n_{1}} + F_{n_2} = 2^{z_1} + 3^{z_2} < \frac{\alpha^{n_1}}{\sqrt{5}} +2^{z_1}.
\end{align}
Using the inequality \eqref{eq34} and the Binet form for $F_{n_1}$, the above inequality can be written as 
\begin{equation}\label{eq45c}
F_{n_2} < \frac{\beta^{n_1}}{\sqrt{5}} + 2^{z_1}< \frac{\beta^{n_1}}{\sqrt{5}} + 2^{0.45 n_1}.
\end{equation}
Since $n_1 >100 $ and $ n_2 \geq 85$, we have $F_{n_2} > \frac{\beta^{n_1}}{\sqrt{5}} + 2^{0.45 n_1}$ and this contradicts the inequality \eqref{eq45c}. Thus, $\Lambda_{1} < 0$ is not possible for $(n_1 - n_2) \leq 15$. 
 
 \smallskip
{\bf Sub-case II.} ($(n_1 - n_2) > 15$).

In this case, one can easily check that 
\[\frac{3\sqrt{5}}{(\al/2^{0.45})^{n_1 - n_2 }}< \frac{1}{2}.\]
Now from \eqref{eq35}  
\begin{equation}
|1 - e^{\Lambda_1}| < \frac{3\sqrt{5}}{(\al/2^{0.45})^{n_1 - n_2}} <  \frac{1}{2}
\end{equation} which  implies that 
\begin{align}\label{eq45d}
-\frac{1}{2} < 1 - e^{\Lambda_1} < \frac{1}{2}
\end{align} and thus, by considering right hand side inequality of \eqref{eq45d}, we have
 $e^{|\Lambda_1|} < 2$. Since $\Lambda_1< 0$ , we have that 
 \[ 
 0 < |\Lambda_1| \leq e^{|\Lambda_1|}-1=e^{|\Lambda_1|}|1 - e^{\Lambda_1}|<\frac{6\sqrt{5}}{(\al/2^{0.45})^{n_1 - n_2}}. \]
 By proceeding in a similar way as in the case of $\Lambda_1 > 0$ with $ |\Lambda_1| = -\Lambda_1 $, we obtain
\begin{equation}\label{eq45}
0 < n_1 \ga - z_2 + \mu < \frac{13}{(\al/2^{0.45})^{n_1 - n_2}}
\end{equation}
where 
\[ \gamma:= \left(\frac{\log \al }{\log 3}\right),\;\; \mu:= \left(-\frac{\log\sqrt{5}}{\log 3}\right),\;\;
 A:= 13,\;\; B:= (\al/2^{0.45}).\]

Here, we also took  $M:= 9\times 10^{30}$ which is an upper bound for $n_1$. To apply Lemma \ref{lem13}, consider $\epsilon:= \|\mu q_{64}\| - M\|\ga q_{64}\|$ which is positive. If \eqref{eq01} has a solution $(n_1, n_2, z_1, z_2)$, then maximum value of $ (n_1-n_2) $  is $484.034\ldots$. 

\smallskip

Next, we look into the equation \eqref{eq38} to estimate the upper bound for $n_1$. Now put 
\[
\Lambda_2:= z_2\log 3 - n_1 \log \alpha + \log \phi(n_1-n_2),\] where we take
$\phi (x)= \sqrt{5}(1+\alpha^{-x})^{-1}$, which implies
$$| 1- e^{\Lambda_2} | <  \frac{3\sqrt{5}}{(\al/2^{0.45})^{n_1}}.$$  
If $\Lambda_2 = 0$, then
\begin{equation}\label{eq46a}
3^{z_2} \sqrt{5}  = \alpha^{n_1}+\alpha^{n_2}
\end{equation}
 By conjugating \eqref{eq46a} and subtracting the resulting equation from \eqref{eq46a},  we get 
\begin{equation}
2\cdot 3^{z_2} = F_{n_1} + F_{n_2} = 2^{z_1} + 3^{z_2} < 2\cdot 3^{z_2}, 
\end{equation} 
which is not true. Thus, we have $\Lambda_2 \neq  0$. Like previous case, we consider the cases separately according as $\Lambda_2 >  0$ and $\Lambda_2 <  0$. First assume that $\Lambda_2 >  0$.  
 By the same line of argument as before, we have  
\begin{equation}\label{eq46}
0< z_2\left(\frac{\log3}{\log\alpha}\right)-n_1 +\left(\frac{\log \phi(n_1-n_2)}{\log \alpha}\right) < 
\frac{14}{(\al/2^{0.45})^{n_1}}.
\end{equation} 
Again we will use here Lemma \ref{lem13} with the following parameters
\begin{equation}\label{eq47}
\gamma:= \frac{\log3}{\log\alpha},\;\; \mu:= \frac{\log \phi(n_1-n_2)}{\log \alpha},\;\; 
 A:= 14,\;\; B:= \al/2^{0.45}.
\end{equation}
Proceeding like before with $M := 3\times 10^{30}$ and applying Lemma \ref{lem13} to the inequality \eqref{eq46} for all possible choices of $n_1-n_2 \in [0, 485] \setminus \lbrace 2 \rbrace $   and we find that, if \eqref{eq01} has a solution $(n_1, n_2, z_1, z_2)$, then $n_1 \in (100,176]$.

For $(n_1-n_2)  = 2$, we have $\epsilon$ is negative. Thus, we can not apply Lemma \ref{lem13}. We use the following arguments to reduce the size of $n_1$ when $(n_1-n_2)  = 2$. For this case, $\phi(2) = \alpha$ and hence the inequality \eqref{eq46} becomes
\[0< z_2\left(\frac{\log3}{\log\alpha}\right)-(n_1 - 1) < 
		\frac{14}{(\al/2^{0.45})^{n_1}}.\]
Using properties of continued fraction of irrational number $\log 3/\log \al$  in the above inequality, we get 
 \begin{equation}\label{eq50}
 \frac{z_2}{q_{63}(q_{64}+q_{63})} < z_2 \left|\frac{\log{3}}{\log{\alpha}} - \frac{p_{63}}{q_{63}} \right| < z_2 \left|\frac{\log{3}}{\log{\alpha}} - \frac{(n_1-1)}{z_2} \right| < \frac{14}{(\alpha/2^{0.45})^{n_1}}, 
 \end{equation}
which implies that 
\begin{align*}
\frac{1}{(q_{64}+q_{63})} < \frac{(14 q_{63}/z_2)}{(\alpha/2^{0.45})^{n_1}}.
\end{align*}
Now using {\sf Mathematica} for $n_1= 9 \times 10^{30} $ and  $ z_2 = 0.45 n_1= 4.05 \times 10^{30}$, we have 
\begin{equation}\label{eq51}
n_1 < \frac{\log \left(14 \cdot q_{63}/z_2\right)}{\log{(\alpha/2^{0.45})}} < 488.
\end{equation}
In a similar way, when $\Lambda_2 < 0$, we obtain
\begin{equation}\label{eq48}
0 < n_1 \ga - z_2 + \mu < \frac{13}{(\al/2^{0.45})^{n_1 - n_2}}
\end{equation}
where 
\[ \gamma:= \left(\frac{\log \al }{\log 3}\right),\;\; \mu:= \left(-\frac{\log \phi(n_1-n_2)}{\log 3}\right),\;\;
 A:= 13,\;\; B:= (\al/2^{0.45}).\]
By applying Lemma \ref{lem13} to the inequality \eqref{eq45} for all possible choices of $n_1-n_2 \in [0, 485] \setminus \lbrace 2 \rbrace$, we find that if \eqref{eq01} has a solution $(n_1, n_2, n_3, z)$, then $n_1 \in (100,179]$.

For $n_1 - n_2 = 2$, by the same line argument, from the inequality \eqref{eq48}, we have
  \begin{equation}
  \frac{(n_1-1)}{q_{64}(q_{65}+q_{64})} < (n_1-1) \left|\frac{\log{\alpha}}{\log{3}} - \frac{p_{64}}{q_{64}} \right| < (n_1-1) \left|\frac{\log{\alpha}}{\log{3}} - \frac{z_2}{(n_1-1)} \right| < \frac{13}{(\alpha/2^{0.45})^{n_1}}.
  \end{equation}
Now using {\sf Mathematica}, we get $n_1 < 493$.  Thus, for all cases we have $n_1 \in (100,  493]$ and hence $z_2 \leq 222$. 

This upper bound of $z_2$ is still very large  for computational purpose. Further, we reduce the size of $z_2$ by using $3$-adic valuation. Here, notice that $\nu_{3}(F_{n_1} + F_{n_2} - 2^{z_1}) = z_2$.  We need to exclude the trivial cases when $F_{n_1} + F_{n_2} - 2^{z_1} = 0$, as the valuation is infinite. Thus, {\sf Mathematica} returns $\nu_{3}(F_{n_1} + F_{n_2} - 2^{z_1}) \leq 12$, 
and the estimate $\al^{n_1 -2} \leq F_{n_1} + F_{n_2} \leq 2\cdot 3^{z_2}$ gives $n_1 \leq 40$. This is a contradiction as $n_1 > 100$ which completes the proof of the Theorem \ref{th2}.

\end{document}